\documentclass[11pt, letterpaper, reqno]{amsart}

\usepackage[a4paper,
            top=30mm,bottom=30mm,
            left=30mm,right=30mm]{geometry}

\usepackage{amsmath, amsthm, amsfonts, amssymb, amscd, mathtools, bm} 
\usepackage{xcolor} 
\usepackage[unicode=true]{hyperref} 
\usepackage[shortlabels]{enumitem} 
\usepackage[T1]{fontenc}
\usepackage{empheq}
\usepackage{placeins}

\theoremstyle{plain}
\newtheorem{theorem}{Theorem}[section]
\newtheorem{lemma}[theorem]{Lemma}

\newtheorem{proposition}[theorem]{Proposition}

\theoremstyle{definition}
\newtheorem{definition}[theorem]{Definition}

\theoremstyle{remark}
\newtheorem{remark}[theorem]{Remark}

\numberwithin{equation}{section}

\newcommand{\dx}{\mathrm{d} x}
\newcommand{\dy}{\mathrm{d} y}
\newcommand{\dt}{\mathrm{d} t}
\newcommand{\io}{\int_\Omega}
\newcommand{\rd}{\mathrm{d}}
\newcommand{\eps}{\varepsilon}
\newcommand{\N}{\mathbb{N}}
\newcommand{\R}{\mathbb{R}}

\newcommand{\br}{\bar{\rho}}
\newcommand{\bn}{\bar{n}}
\newcommand{\bu}{\bar{u}}
\newcommand{\bv}{\bar{v}}

\title[]{High-friction limit for bipolar Euler--Riesz systems}

\author[]{Nuno J. Alves}
\author[]{Jan Haskovec}

\address[N. J. Alves]{
      University of Vienna, Faculty of Mathematics, Oskar-Morgenstern-Platz 1, 1090 Vienna, Austria.}
\email{nuno.januario.alves@univie.ac.at}

\address[J. Haskovec]{
        King Abdullah University of Science and Technology, CEMSE Division, Thuwal 23955-6900, Kingdom of Saudi Arabia.}
\email{jan.haskovec@kaust.edu.sa}

\begin{document}

\begin{abstract}
We consider a bipolar Euler--Riesz system and rigorously justify the high-friction limit of weak solutions towards a bipolar aggregation-diffusion system with Riesz interactions. The analysis is carried out via the relative entropy method in the regime where smooth solutions of the limiting equations exist. This extends previous results on the high-friction limit of bipolar Euler--Poisson systems to a more general class of interactions, and extends the one-species Euler--Riesz case to the bipolar setting.
\end{abstract}

\keywords{Bipolar Euler--Riesz equations, high-friction limit, bipolar aggregation-diffusion, relative entropy method}
\subjclass[2020]{35Q31, 35Q35} 
\maketitle
\thispagestyle{empty} 

\section{Introduction}
Let $d \in \N$ denote the spatial dimension, and let $0 < \alpha < d$. We consider the following bipolar Euler--Riesz system
\begin{equation} \label{BER}
\begin{dcases}
\partial_t \rho + \nabla \cdot (\rho u) = 0, \\ 
\partial_t(\rho u) + \nabla \cdot (\rho u \otimes u) + \nabla p_1(\rho) + \sigma \rho \nabla K_\alpha \ast (\rho - n) = - \zeta \rho u, \\
\partial_t n + \nabla \cdot (n v) = 0, \\
\partial_t(n v) + \nabla \cdot (n v \otimes v) + \nabla p_2(n) - \sigma n \nabla K_\alpha \ast (\rho - n) = - \zeta n v,
\end{dcases}
\end{equation}
posed in $(0,T) \times \Omega$, where $0 < T < \infty$ and $\Omega \subseteq \R^d$ is a bounded domain with smooth boundary $\partial \Omega$.  
System \eqref{BER} consists of two continuity equations for the densities $\rho$ and $n$
of positively and negatively charged ions, respectively, and two momentum equations for their corresponding velocities $u$ and $v$.
We prescribe the no-flux boundary conditions
\begin{equation} \label{boundary_cond}
u \cdot \nu = v \cdot \nu = 0 \qquad \text{on } [0,T) \times \partial \Omega,
\end{equation}
where $\nu$ denotes the outward unit normal to the boundary $\partial\Omega$.
The equations are coupled through interaction forces modeled by the fractional Riesz kernel
\begin{equation} \label{kernel}
K_\alpha(x) = \tfrac{1}{d-\alpha} |x|^{\alpha - d},
\end{equation}
with interaction strength $\sigma > 0$. The pressures are assumed to follow the power laws
\begin{equation} \label{pressures}
p_1(\rho) = \rho^{\gamma_1}, \quad p_2(n) = n^{\gamma_2}, \quad \gamma = \min\{\gamma_1, \gamma_2 \} > 1,
\end{equation}
with adiabatic exponents $\gamma_1$, $\gamma_2$.  

The bipolar Euler--Riesz system \eqref{BER} describes the dynamics of a compressible fluid composed of two interacting species, generalizing the well-known bipolar Euler--Poisson system to a broader class of interactions. The latter serves as a basic model in plasma physics and semiconductor theory \cite{chen1984introduction, markowich2012semi}. Indeed, for $d \geq 3$ and $\alpha = 2$, the kernel $K_2$ corresponds (up to a multiplicative constant) to the Coulomb kernel, whose convolution with $\rho-n$ in $\R^d$ yields the solution of the Poisson equation $-\Delta \phi = \rho - n$. The counterpart of the Coulomb kernel in bounded domains is the Neumann function, which is controlled in the same manner as $K_2$;
see, e.g., \cite{kenig1994harmonic}. Another natural extension is to include magnetic effects by coupling each species with Maxwell's equations, leading to the bipolar Euler--Maxwell system; see \cite{guo2016global, alves2024relative}. We may therefore interpret $\rho$, $u$ as the density and velocity of positively charged ions, and $n$, $v$ as the corresponding quantities for negatively charged ions. The constant $\zeta > 0$ denotes the collision frequency between particles.  

The purpose of this work is to establish the high-friction limit $\zeta \to\infty$ of weak solutions of a diffusively scaled version of \eqref{BER},
see system~\eqref{BER_HF} below,
towards strong solutions of the bipolar aggregation-diffusion system with Riesz interactions:
\begin{equation} \label{BAD}
\begin{dcases}
\partial_t \rho = \nabla \cdot \big( \nabla p_1(\rho) + \sigma \rho \nabla K_\alpha \ast (\rho - n) \big), \\
\partial_t n = \nabla \cdot \big( \nabla p_2(n) - \sigma n \nabla K_\alpha \ast (\rho - n) \big).
\end{dcases}
\end{equation}
\par 
Our approach relies on the relative entropy method, a powerful framework for studying stability and asymptotic limits in hyperbolic systems. The key idea is to derive a relative entropy inequality that compares a weak solution of the original system with a strong solution of the limit system. The terms on the right-hand side of this inequality are estimated in terms of the relative entropy itself, allowing the use of Gronwall's lemma to obtain stability, from which convergence follows. We establish our result for adiabatic exponents satisfying
\begin{equation*} 
\gamma \geq 2 \quad \text{and} \quad 1 < \alpha < \frac{d}{2} + 1,
\end{equation*}
or
\begin{equation*} 
2 - \frac{\alpha - 1}{d} \leq \gamma < 2 \quad \text{and} \quad 1 < \alpha \leq \frac{d}{2} + 1,
\end{equation*}
which should be contrasted with the less restrictive assumptions $\gamma \geq 2 - \alpha/d$, $1 < \alpha < d$ obtained in the one-species case \cite{alves2024weak}. In particular, our analysis indicates that in two-species models the admissible range of adiabatic exponents is more restrictive, since the techniques that apply in the one-species case break down in the presence of asymmetric coupling between the densities.
\par This extends previous results on the high-friction limit of one-species Euler--Riesz systems~\cite{alves2024weak} and two-species Euler--Poisson systems~\cite{alves2022relaxation}.
The relative entropy method has also been applied successfully to other singular limits such as the zero-electron-mass and quasi-neutral limits in bipolar Euler--Poisson systems \cite{alves2024zero}. For broader perspectives on the relative entropy method we refer to the foundational work \cite{dafermos1979second} and to subsequent studies \cite{lattanzio2013relative, mielke2015uniform, lattanzio2017gas, giesselmann2017relative, haskovec2018decay, huo2019high, carrillo2020relative, georgiadis2023asymptotic, alves2024role, georgiadis2024alignment, elbar2025nonlocal}.  

To the best of our knowledge, this is the first appearance of both system \eqref{BER} and its formal limit \eqref{BAD} in the literature, whether considered separately or connected via the high-friction limit. The one-species counterpart of \eqref{BER} has attracted considerable attention in recent years; see \cite{choi2021relaxation, choi2022well, choi2024damped, choi2025global} for results on well-posedness and relaxation. As for the gradient flow system~\eqref{BAD}, its one-species variant was recently studied in \cite{huang2024nonlinear}, and a bipolar version with Coulomb interactions in \cite{kinderlehrer2017wasserstein}.  

We emphasize that, while our main result, Theorem~\ref{thm_main}, is rigorous in its derivation, it remains formal in scope, as we do not address the existence of solutions. For the Euler--Riesz system~\eqref{BER}, we consider weak solutions whose regularity is consistent with the natural energy estimates; see Sections~\ref{section_energy} and \ref{section_solutions}. For the limit system~\eqref{BAD}, we hypothesize that existence of solutions could be established by adapting the techniques of \cite{kinderlehrer2017wasserstein, huang2024nonlinear} to the bipolar setting with Riesz interactions.  

We organize the manuscript as follows. Section~\ref{section_energy} presents the energy and relative energy identities for system~\eqref{BER}. Section~\ref{section_solutions} introduces the notion of weak solutions used in our analysis. The main theorem is stated in Section~\ref{section_main} and proved in Section~\ref{section_proof}.

\section{Energy and relative energy} \label{section_energy}
In this section, we present the energy and relative energy functionals associated with system~\eqref{BER}; the latter being the basis of the method we use to establish the high-friction limit. \par 
Let $(\rho,u,n,v)$ be a smooth solution of \eqref{BER} satisfying the no-flux boundary conditions~\eqref{boundary_cond}. Taking the inner product of the first and second momentum equations by $u$ and $v$, respectively, and integrating over $\Omega$ we obtain
\[\frac{\rd}{\dt} \io \tfrac{1}{2} \rho |u|^2 + h_1(\rho) \, \dx + \sigma \io (\partial_t \rho) K_\alpha \ast (\rho - n) \, \dx = -\zeta \io \rho |u|^2 \, \dx,\]
and 
\[\frac{\rd}{\dt} \io  \tfrac{1}{2} n |v|^2 + h_2(n) \, \dx - \sigma \io (\partial_t n) K_\alpha \ast (\rho - n) \, \dx= -\zeta \io n |v|^2 \, \dx,\]
where $h_1,h_2$ are the internal energy functions related to the pressures through
\[\rho h_1^{\prime \prime}(\rho) = p_1^\prime(\rho), \quad  n h_2^{\prime \prime}(n) = p_1^\prime(n)\]
and therefore given by 
\begin{equation} \label{internal_energy}
h_1(\rho) = \tfrac{1}{\gamma_1 - 1}\rho^{\gamma_1}, \quad h_2(n) = \tfrac{1}{\gamma_2 - 1}n^{\gamma_2}.
\end{equation}
Now, we use the symmetry of the kernel $K_\alpha$ to deduce that
\[\io (\partial_t \rho) K_\alpha \ast (\rho - n) \, \dx - \io (\partial_t n) K_\alpha \ast (\rho - n) \, \dx = \frac{\rd}{\dt} \io \tfrac{1}{2} (\rho - n) K_\alpha \ast (\rho - n) \, \dx \]
which yields, by adding the two expressions above, the energy identity for \eqref{BER},
\begin{equation} \label{energy_identity0}
\begin{split}
\frac{\rd}{\dt} \io \tfrac{1}{2} &\rho |u|^2 + \tfrac{1}{2} n |v|^2  +h_1(\rho)+ h_2(n) + \sigma \tfrac{1}{2} (\rho - n) K_\alpha \ast (\rho - n) \, \dx  \\
& = - \zeta \io \rho |u|^2  +n |v|^2 \, \dx.
\end{split}
\end{equation}
\par 
The interaction energy term $\int_\Omega \sigma \tfrac{1}{2} (\rho - n) K_\alpha \ast (\rho - n) \, \dx$ represents the electrostatic energy associated with the charge density $\rho - n$, generalized to the Riesz kernel $K_\alpha$. For $d \geq 3$ and $\alpha = 2$, this reduces to the classical Coulomb energy $\io \tfrac{1}{2} (\rho - n) \phi \, \dx$ with $- \Delta \phi = \rho -n$.
\par
Identity \eqref{energy_identity0} represents the conservation ($\zeta = 0$) or dissipation ($\zeta > 0$) of total energy for solutions to \eqref{BER}, indicating as well their a priori regularity.  In the one-species Euler--Riesz setting, these a priori estimates for the density can be further improved. This has been achieved, for instance, via compensated integrability in \cite{agt}, and in the context of spherically symmetric solutions in \cite{cccy}.  \par 
We now consider the kinetic and potential energy functionals, denoted by $\mathcal{K}$ and $\mathcal{E}$, respectively, and given by
\begin{equation} \label{kinetic_energy_funct}
\mathcal{K}(\rho,\rho u, n ,n v) = \io \tfrac{1}{2} \rho |u|^2 + \tfrac{1}{2} n |v|^2  \, \dx,
\end{equation}
\begin{equation} \label{potential_energy_funct}
\mathcal{E}(\rho,n) = \io h_1(\rho)+ h_2(n) + \sigma \tfrac{1}{2} (\rho - n) K_\alpha \ast (\rho - n) \, \dx.
\end{equation}
Note that the kinetic energy functional is seen as a functional of the momenta $\rho u$ and $nv$.

We now present the relative kinetic and relative potential energy functionals. Given another smooth solution $(\br,\bu,\bn,\bv)$ of \eqref{BER}, the relative kinetic (resp. potential) energy functional is obtained as the quadratic part of the Taylor expansion of the kinetic (resp. potential) energy functional around $(\br,\bu,\bn,\bv)$. In particular,
\begin{align*}
\mathcal{K}(\rho,\rho u, n ,n v \, | \, \br, \br \bu,\bn, \bn \bv ) = & \  \mathcal{K}(\rho,\rho u, n ,n v ) - \mathcal{K}(\br, \br \bu,\bn, \bn \bv ) \\
& - \io \frac{\delta \mathcal{K}}{\delta \rho}(\br, \br \bu,\bn, \bn \bv) (\rho - \br) \, \dx \\
& - \io \frac{\delta \mathcal{K}}{\delta n}(\br, \br \bu,\bn, \bn \bv) (n - \bn) \, \dx  \\
& - \io \frac{\delta \mathcal{K}}{\delta (\rho u)}(\br, \br \bu,\bn, \bn \bv) \cdot (\rho u - \br \bu) \, \dx \\
& - \io \frac{\delta \mathcal{K}}{\delta (nv)}(\br, \br \bu,\bn, \bn \bv) \cdot (nv - \bn \bv) \, \dx
\end{align*}
where $\delta \mathcal{K} / \delta \rho$, $\delta \mathcal{K} / \delta n $, $\delta \mathcal{K} / \delta (\rho u)$, $\delta \mathcal{K} / \delta (nv)$ are the functional derivatives of $\mathcal{K}$. A straightforward calculation gives 
\begin{equation} \label{relative_kinetic}
\mathcal{K}(\rho,\rho u, n ,n v \, | \, \br, \br \bu,\bn, \bn \bv ) = \io \tfrac{1}{2} \rho |u - \bu|^2 + \tfrac{1}{2} n |v - \bv|^2 \, \dx.
\end{equation}
Similarly, we have the relative potential energy
\begin{equation}\label{relative_potential}
\begin{split}
\mathcal{E}(\rho, n \, | \, \br,\bn ) & =   \mathcal{E}(\rho, n ) - \mathcal{E}(\br,\bn ) - \io \frac{\delta \mathcal{E}}{\delta \rho}(\br,\bn) (\rho - \br) \, \dx  - \io \frac{\delta \mathcal{E}}{\delta n}(\br, \bn) (n - \bn) \, \dx  \\ 
& = \io h_1(\rho | \br ) + h_2(n|\bn) + \sigma \tfrac{1}{2}(\rho - \br - n + \bn) K_\alpha \ast (\rho - \br - n + \bn) \, \dx,
\end{split}
\end{equation}
where for a function $h=h(r)$ the relative quantity $h(r | \bar r)$ is given by
its Taylor expansion to second order,
 \[h(r | \bar r) = h(r) - h(\bar r) - h^\prime(\bar r)(r - \bar r).\]
\par 
Adding the relative kinetic and relative potential energy functionals results in the relative total energy. Taking its time derivative yields the relative energy identity, which can be computed using the abstract formalism presented in \cite{giesselmann2017relative} and adapting it to the bipolar case as in \cite{alves2022relaxation} (see also \cite{alves2024weak, alves2024zero}). The relative energy identity for the present case reads
\begin{equation} \label{relative_energy_identity0}
\begin{split}
\frac{\rd}{\dt} \Big(\mathcal{K}(\rho,\rho u, n ,n v \, &| \, \br, \br \bu,\bn, \bn \bv ) +  \mathcal{E}(\rho, n \, | \, \br,\bn )  \Big) +  \zeta \io \rho |u - \bu|^2  + n |n - \bn|^2 \, \dx \\
 = &  - \io \nabla \bu : \rho (u - \bu) \otimes (u - \bu) + \nabla \bv : \rho (v - \bv) \otimes (v - \bv) \, \dx \\ 
    & - \io (\nabla \cdot \bu)p_1(\rho| \br) + (\nabla \cdot \bv)p_2(n| \bn) \, \dx \\
    & + \sigma \io \big((\rho - \br)\bu - (n-\bn)\bv \big) \cdot \nabla K_\alpha \ast(\rho - \br - n + \bn)  \, \dx.
\end{split}
\end{equation}
\par
The calculations that led to identity~\eqref{relative_energy_identity0} are under the formal assumptions that the involved functions are smooth and satisfy equations~\eqref{BER} in a classical way. The purpose of this work is to derive an analogous expression for solutions where the regularity is only given by the energy identity, and rigorously justify the asymptotics $\zeta \to \infty$ for that class of solutions. The solutions of the limiting system are however assumed to have enough regularity so that the right-hand-side of \eqref{BER} can be controlled in terms of the relative total energy.

\section{Weak solutions to the bipolar Euler--Riesz system} \label{section_solutions}
This section presents the notion of weak solutions to a scaled version of system~\eqref{BER} for which we establish the high-friction limit result. We begin by observing that in limit $\zeta \to \infty$ of \eqref{BER}, one obtains $u = v = 0$ and hence the system reduces to an equilibrium state. To capture a nonequilibrium limiting phenomenon we consider the diffusive scaling
\[\tilde{\rho}(t,x) = \rho(t/\sqrt{\eps}, x), \quad \tilde{n}(t,x) = n(t/\sqrt{\eps}, x) \]
and 
\[\tilde{u}(t,x) = (1/\sqrt{\eps}) u(t/\sqrt{\eps}, x), \quad \tilde{v}(t,x) = (1/\sqrt{\eps}) v(t/\sqrt{\eps}, x) \]
where $\eps = 1/\zeta^2$. Rewriting system~\eqref{BER} in terms of the rescaled quantities, dropping the tilde notation, leads to 
\begin{equation} \label{BER_HF}
\begin{dcases}
\partial_t \rho + \nabla \cdot (\rho u) = 0, \\
\eps \big( \partial_t(\rho u) + \nabla \cdot (\rho u \otimes u) \big) + \nabla p_1(\rho) + \sigma \rho \nabla K_\alpha \ast (\rho - n) = - \rho u, \\
\partial_t n + \nabla \cdot (n v) = 0, \\
\eps \big(\partial_t(n v) + \nabla \cdot (n v \otimes v) \big) + \nabla p_2(n) - \sigma n \nabla K_\alpha \ast (\rho - n) = - n v.
\end{dcases}
\end{equation}
The high-friction limit then corresponds to the passage $\eps \to 0$ in system~\eqref{BER_HF}, which formally yields the bipolar aggregation-diffusion system~\eqref{BAD}. \par 
We now describe the framework of weak solutions to \eqref{BER_HF} that will be used in the sequel. For notational convenience, we introduce the (local) total energy functional $\mathcal{H}$, defined by
\begin{equation} \label{total_energy_HF}
\mathcal{H}(\rho, u, n ,v) = \eps \tfrac{1}{2}\rho|u|^2 + \eps \tfrac{1}{2}n|v|^2 + h_1(\rho) + h_2(n) + \sigma \tfrac{1}{2} (\rho - n) K_\alpha \ast (\rho - n).
\end{equation}
Moreover, system~\eqref{BER_HF} is supplemented with initial data $(\rho_0,u_0,n_0,v_0)$.
\begin{definition} \label{def_weak_sol}
 A tuple of functions $(\rho, u,n,v)$ with $\rho, n \ge 0$ and regularity
\[\rho \in C\big([0,T);  L^{\gamma_1}(\Omega) \big), \quad n \in C\big([0,T);  L^{\gamma_2}(\Omega) \big), \quad \gamma = \min\{\gamma_1,\gamma_2 \} > 1, \] 
 \[\rho u, nv \in C\big([0,T);L^1(\Omega,\R^d)\big), \quad \rho |u|^2, n|v|^2 \in  C\big([0,T);L^1(\Omega)\big), \]
is a dissipative weak solution of (\ref{BER_HF}) provided that:
\begin{enumerate}[(i)]
 \item $(\rho, u,n,v)$ satisfies (\ref{BER_HF}) in the following weak sense:
 \begin{equation} \label{weak1}
         \int_0^T \io \rho \, \partial_t \varphi \, \dx \, \dt+\int_0^T \io  \rho u \cdot \nabla \varphi \, \dx \, \dt + \io \rho_0  \varphi_{|_{t=0}}\, \dx=0,
        \end{equation} \\
 \begin{equation} \label{weak2}
        \begin{split}
            \int_0^T & \io \eps  \rho u \cdot \partial_t \tilde{\varphi}\, \dx \, \dt + \int_0^T \io  \eps  \rho u \otimes u : \nabla \tilde{\varphi} \, \dx \, \dt + \int_0^T \io p_1(\rho)  \nabla \cdot \tilde{\varphi} \, \dx \, \dt \\
       - & \int_0^T \io \sigma \rho \big( \nabla K_\alpha \ast(\rho - n) \big) \cdot \tilde{\varphi} \, \dx \, \dt + \io  \eps \rho_0 u_0 \cdot \tilde{\varphi}_{|_{t=0}}\, \dx \\
       & =   \int_0^T \io  \rho u \cdot \tilde{\varphi} \, \dx \, \dt ,
        \end{split}
        \end{equation} \\
\begin{equation} \label{weak3}
\int_0^T \io n \, \partial_t \psi \, \dx \, \dt
+ \int_0^T \io n v \cdot \nabla \psi \, \dx \, \dt
+ \io n_0  \psi _{|_{t=0}}\, \dx = 0,
\end{equation} \\
\begin{equation} \label{weak4}
        \begin{split}
            \int_0^T & \io \eps  n v \cdot \partial_t \tilde{\psi}\, \dx \, \dt + \int_0^T \io  \eps  n v \otimes v : \nabla \tilde{\psi} \, \dx \, \dt + \int_0^T \io p_2(n)  \nabla \cdot \tilde{\psi} \, \dx \, \dt \\
       + & \int_0^T \io \sigma n \big( \nabla K_\alpha \ast(n - \rho) \big) \cdot \tilde{\psi} \, \dx \, \dt + \io  \eps n_0 v_0 \cdot \tilde{\psi}_{|_{t=0}}\, \dx \\
       & =   \int_0^T \io  n v \cdot \tilde{\psi} \, \dx \, \dt ,
        \end{split}
\end{equation}
      for all Lipschitz test functions $\varphi, \psi : [0,T) \times \bar{\Omega} \to \R, \ \tilde{\varphi}, \tilde{\psi}:\mathopen{[}0,T) \times \bar{\Omega} \to \R^d$ compactly supported in time and satisfying $\tilde{\varphi} \cdot \nu = \tilde{\psi} \cdot \nu = 0 $ on $\mathopen{[}0,T) \times \partial \Omega,$ where $\nu$ is the outward unit normal to $\partial \Omega$.
    \item $(\rho,u,n,v)$ conserves mass:
    \begin{equation} \label{massconservation}
 \io \rho \, \dx = M_1, \quad \io n \, \dx = M_2,
  \end{equation}
  for some $M_1,M_2 \in\R$ and all $t \in [0,T)$.
   \item $(\rho, u,n,v)$ has finite total energy:
  \begin{equation} \label{energyconservation}
  \io \mathcal{H}(\rho,u,n,v) \, \dx \leq C_0
 \end{equation}
 for some $C_0 \in\R$ and all $t \in [0,T)$.
 \item $(\rho, u,n,v)$ satisfies the following weak form of the energy dissipation:
\begin{equation} \label{weakdissip}
   \begin{split}
- & \int_0^T   \io \mathcal{H}(\rho,u,n,v) \, \dot{\theta}\, \dx \, \dt +  \int_0^T \io \big( \rho |u|^2+ n |v|^2\big) \theta \, \dx \, \dt \\
& \leq   \io  \mathcal{H}(\rho_0,u_0,n_0,v_0) \, \theta(0) \, \dx 
   \end{split}
  \end{equation} \\
  for all nonnegative and compactly supported $\theta \in W^{1,\infty}([0,T))$. 
 \end{enumerate}
\end{definition}
 
Let us provide a few clarifying remarks on the properties of a dissipative weak solution $(\rho,u,n,v)$ as in Definition~\ref{def_weak_sol}. We begin by noting that the assumed regularity is sufficient to guarantee that all terms in the weak formulations of points (i) and (iv) are well defined, except for those involving the kernel $K_\alpha$. In particular, we need to make sense of the following terms:
 \begin{equation} \label{terms_1}
 \io \rho \big( \nabla K_\alpha \ast(\rho - n) \big) \cdot \tilde{\varphi} \, \dx, \quad \io n \big( \nabla K_\alpha \ast(\rho - n) \big) \cdot \tilde{\psi} \, \dx,
\end{equation}
 and 
 \begin{equation} \label{term_2}
 \io (\rho - n) K_\alpha \ast (\rho - n) \, \dx
 \end{equation} 
 where $\tilde{\varphi},\tilde{\psi}$ are Lipschitz continuous on $\bar{\Omega}$. Since $\Omega$ is bounded, we have $\rho, n \in C\big([0,T); L^\gamma(\Omega) \big)$. Moreover, we extend all the involved functions by zero to all of $\R^d$, so that we can write, for $x \in \Omega$,
\begin{align*}
K_\alpha \ast (\rho - n)(x) & = \tfrac{1}{d-\alpha} \io (\rho(y) - n(y))|x-y|^{\alpha - d} \, \dy \\
&  =  \tfrac{1}{d-\alpha} \int_{\R^d} (\rho(y) - n(y))|x-y|^{\alpha - d} \, \dy
\end{align*}
and similarly for $\nabla K_\alpha \ast (\rho - n)$. We first give meaning to the weak gradient of $K_\alpha \ast (\rho - n)$. This is the content of the next proposition, whose proof is a straightforward adaptation of the proofs of \cite[Proposition 3.1]{alves2024role} and \cite[Proposition 4.1]{alves2024weak}.

\begin{proposition} \label{prop_nabla_K}
Let $d > 1$ and $1 < \alpha < d$. If $f,g \in C\big([0,T); L^\gamma(\Omega) \big)$ with $\gamma \geq 2d/(d+\alpha)$, then $K_\alpha \ast f \in C\big([0,T); L^{\frac{2d}{d-\alpha}}(\Omega) \big)$ and $g (K_\alpha \ast f) \in C\big([0,T); L^{1}(\Omega) \big)$. Moreover, the weak spatial gradient of $K_\alpha \ast f$ is given by 
\[\nabla (K_\alpha \ast f)(t,x) = \nabla K_\alpha \ast f(t,x) = - \io f(t,y) (x-y) |x-y|^{\alpha - d - 2} \, \dy\]
and belongs to $C\big([0,T); L^{\frac{2d}{d-\alpha+2}}(\Omega) \big)$. Therefore,
\[|\nabla K_\alpha| \leq c K_{\alpha - 1} \]
where $c = d-\alpha -1$.
\end{proposition}
 
 The previous proposition, together with the regularity of a dissipative weak solution, yields that for $\gamma \geq 2d/(d+\alpha)$, the term in \eqref{term_2} is well defined. It remains to show that the same holds true for the terms in~\eqref{terms_1}. However, the presence of the mixed term $\rho - n$ makes this task more complicated, and here we observe a fundamental difference between the one-species and two-species models. It would be desirable to make sense of the terms in \eqref{terms_1} without any further restriction on the adiabatic exponents. This is possible in the one-species model, where mixed densities such as $\rho - n$ are not present. As an illustration, note that, by the spatial antisymmetry of $\nabla K_\alpha$, we have
 \begin{equation} \label{nabla_controlled}
 \begin{split}
 \io \rho \big( \nabla K_\alpha \ast\rho \big) \cdot \tilde{\varphi} \, \dx   & = \io \io \tilde{\varphi}(x) \cdot \rho(x) \nabla K_\alpha(x-y) \rho(y) \, \dx \, \dy \\ 
 & = \frac{1}{2} \io \io (\tilde{\varphi}(x) - \tilde{\varphi}(y) ) \cdot \rho(x) \nabla K_\alpha(x-y) \rho(y) \, \dx \, \dy \\
 & \leq \frac{1}{2} \io \io |\tilde{\varphi}(x) - \tilde{\varphi}(y)| \rho(x) |\nabla K_\alpha(x-y)| \rho(y) \, \dx \, \dy \\
 & \leq \frac{1}{2} \io \io \frac{|\tilde{\varphi}(x) - \tilde{\varphi}(y)|}{|x-y|} \rho(x) |x - y|^{\alpha - d} \rho(y) \, \dx \, \dy \\
 & \leq \frac{d-\alpha}{2} \|\nabla \tilde{\varphi} \|_\infty \io \rho (K_\alpha \ast \rho) \, \dx,
 \end{split}
 \end{equation}
 where we have used the Lipschitz continuity of $\tilde{\varphi}$. Unfortunately, the same argument does not hold for the terms in~\eqref{terms_1}. We are therefore forced to restrict our study to the case $\gamma \geq 2d/(d+\alpha - 1)$, which by Proposition~\ref{prop_nabla_K}, with $\alpha$ replaced by $\alpha-1$, yields that
\[\rho \nabla K_\alpha \ast(\rho - n), \ n \nabla K_\alpha \ast(\rho - n) \in C\big([0,T); L^{1}(\Omega) \big) \] 
making the terms in~\eqref{terms_1} well-defined. \par 
 
\section{Statement of the main result} \label{section_main}
In this section we state our main result. It concerns a stability estimate for dissipative weak solutions of the (scaled) bipolar Euler--Riesz equations \eqref{BER_HF} and strong solutions of the bipolar aggregation-diffusion system~\eqref{BAD}. \par
The strategy is to consider system~\eqref{BAD} as an approximation of system~\eqref{BER_HF} so that the relative total energy of \eqref{BER_HF} can be used to compare solutions of \eqref{BER_HF} with solutions of \eqref{BAD} in the limit $\eps \to 0$. \par
Given a solution $(\br,\bn)$ of \eqref{BAD}, we let $\bu$ and $\bv$ be the auxiliary linear velocities given by
\begin{equation} \label{approx_velocities_HF}
\bu = \nabla h_1^\prime (\br) +  \sigma \nabla K_\alpha \ast (\br - \bn), \quad \bv = \nabla h_2^\prime (\bn) -  \sigma\nabla K_\alpha \ast (\br - \bn),
\end{equation} 
so that system~\eqref{BAD} can be rewritten as 
\begin{equation} \label{BAD_approx}
\begin{dcases}
\partial_t \br + \nabla \cdot (\br \bu) = 0, \\
\eps \big( \partial_t(\br \bu) + \nabla \cdot (\br \bu \otimes \bu) \big) + \nabla p_1(\br) + \sigma \br \nabla K_\alpha \ast (\br - \bn) = - \br \bu + \eps \bar{e}_1, \\
\partial_t \bn + \nabla \cdot (\bn \bv) = 0, \\
\eps \big(\partial_t(\bn \bv) + \nabla \cdot (\bn \bv \otimes \bv) \big) + \nabla p_2(\bn) - \sigma \bn \nabla K_\alpha \ast (\br - \bn) = - \bn \bv + \eps \bar{e}_2,
\end{dcases}
\end{equation}
with the approximation errors
  \[\bar{e}_1 = \partial_t(\br \bu) + \nabla \cdot (\br \bu \otimes \bu) , \quad  \bar{e}_2=\partial_t(\bn \bv) + \nabla \cdot (\bn \bv \otimes \bv). \] 
We identify a solution $(\br,\bn)$ of~\eqref{BAD} with the tuple $(\br, \bu, \bn, \bv)$ solving~\eqref{BAD_approx}, where the linear velocities $\bu,\bv$ are defined as in \eqref{approx_velocities_HF}.
 
 We assume that the strong solutions are sufficiently regular classical solutions with bounded first-order derivatives. We also require the corresponding auxiliary linear velocities $\bu,\bv$ to satisfy the no-flux boundary conditions~\eqref{boundary_cond}. Moreover, we impose the extra assumption that the densities $\br$, $\bn$ are bounded away from zero, that is, 
 \begin{equation} \label{bounded_away_zero}
 \bar{\delta} \leq \br, \bn \leq \bar{M} 
 \end{equation}
 for some $\bar{\delta} > 0$ and $\bar{M} < \infty$. \par 
A solution $(\rho,u,n,v)$ of~\eqref{BER_HF} is then compared to a solution $(\br, \bu, \bn, \bv)$ of~\eqref{BAD} through the relative energy $\Psi: [0,T) \to \mathbb{R}$ given by
\begin{equation} \label{Psi}
\Psi(t) = \io \eps \, \mathcal{K}(\rho,\rho u, n ,n v \, | \, \br, \br \bu,\bn, \bn \bv ) +  \mathcal{E}(\rho, n \, | \, \br,\bn ) \, \dx
\end{equation}
 where $\mathcal{K}$ and $\mathcal{E}$ are as in~\eqref{relative_kinetic} and \eqref{relative_potential}, respectively. We note that solutions of~\eqref{BER_HF} depend on $\eps$, and so does $\Psi$. In the sequel we omit an explicit notation of this dependence for simplicity.
 
Within this framework of solutions, we prove the following: 
\begin{theorem} \label{thm_main}
Let $1 < \alpha < d$, and let $(\rho, u, n ,v)$ be a dissipative weak solution to~\eqref{BER_HF} with $\sigma$ sufficiently small and $\gamma > 1$. If either
\begin{equation} \label{gamma_hyp_1}
\gamma \geq 2 \quad \text{and} \quad 1 < \alpha < \frac{d}{2} + 1,
\end{equation}
or
\begin{equation} \label{gamma_hyp_2}
2 - \frac{\alpha - 1}{d} \leq \gamma < 2 \quad \text{and} \quad 1 < \alpha \leq \frac{d}{2} + 1,
\end{equation}
and $(\br,\bn)$ is a strong solution of~\eqref{BAD} satisfying~\eqref{bounded_away_zero}, then the relative energy $\Psi$ of those solutions satisfies
\begin{equation} \label{main_stability}
\sup_{t \in (0,T)}\Psi(t) \leq e^{CT}(\Psi(0) + \eps^2)
\end{equation}
for some $C>0$. Consequently, if $\Psi(0) \to 0$ as $\eps \to 0$, then
\[ \sup_{t \in (0,T)}\Psi(t) \to 0 \quad \text{as} \ \eps \to 0.\]
\end{theorem}
Let us make a few remarks regarding the statement of Theorem~\ref{thm_main}. The condition on~$\sigma$ refers to the fact that it must be chosen small enough so that the relative potential energy between the solutions under consideration is nonnegative. This implies that the relative total energy is nonnegative, thereby serving as a yardstick for comparing the solutions. Analytically, this smallness assumption is a technical requirement guaranteeing that the nonlocal Riesz interaction term can be controlled by the pressure contributions in the relative energy. From a physical standpoint, it corresponds to a weakly coupled regime, where the interaction forces remain subordinate to the pressure effects. 

 \par
Regarding the conditions on~$\gamma$ in~\eqref{gamma_hyp_2}, they are more restrictive than those imposed by the definition of dissipative weak solutions discussed in the previous section. This further constraint will become apparent in the proof of the main theorem in the next section. These technical assumptions on $\gamma$ ensure sufficient integrability of the densities to control the nonlocal interaction terms via the Hardy--Littlewood--Sobolev inequality; see Section~\ref{section_proof_thm}. Moreover, for $\alpha = 2$ we recover the bipolar Euler--Poisson case, with threshold $\gamma \geq 2 - 1/d$; see \cite[Theorem~3.3]{alves2022relaxation}. The condition on~$\gamma$ achieved here is thus consistent with the bipolar Euler--Poisson case. 
\par
Theorem~\ref{thm_main} implies that if the considered solutions coincide at $t = 0$ (uniformly in~$\eps$), then the dissipative weak solution to~\eqref{BER_HF} converges, in the relative energy sense, to the strong solution of~\eqref{BAD}. Behind this result thus lies a weak-strong uniqueness property for the bipolar Euler--Riesz system, with the approximate system~\eqref{BAD_approx} in that case being in fact the exact system~\eqref{BER_HF}.
\begin{remark}
Throughout this work, we consider pressure functions of power-law type, as in \eqref{pressures}, and take the interaction coefficient $\sigma$ to be positive. These assumptions are made for simplicity. However, as shown in \cite{alves2022relaxation}, our main result extends to more general pressure laws with only minor modifications to the analysis. In addition, the case $\sigma < 0$, corresponding to repulsive interactions, can also be treated within the same framework.
\end{remark}

\section{Convergence in the high-friction limit} \label{section_proof}
In this section we provide a proof of Theorem~\ref{thm_main}. We start with two crucial lemmas concerning the internal energy functions, yielding lower bounds for the associated relative entropy. \par 
The positive constants~$C$ that will appear below can change from line to line and may depend on the parameters of the problem but do not depend on the relaxation coefficient~$\varepsilon$.
\subsection{Results on the internal energy}

The first lemma is a particular case of \cite[Lemma~2.4]{lattanzio2013relative}.

\begin{lemma} \label{lemma_h}
Let $h(r) := \tfrac{1}{\gamma - 1} r^\gamma$ for some $\gamma > 1$, and let $\bar r \in [\bar \delta, \bar M]$ be fixed, where $\bar \delta > 0$ and $\bar M < \infty$. Then, there exist $R > \bar M + 1$ and $C > 0$ such that 
\begin{equation*}
h(r | \bar r) \geq \begin{cases}
C |r - \bar r|^2 \quad \text{if} \ r \in [0,R], \\
C |r - \bar r|^\gamma \quad \text{if} \ r \in (R,\infty).
\end{cases}
\end{equation*}
Moreover, if $\gamma \geq 2$, then
\[h(r | \bar r) \geq C|r - \bar r|^2 \]
for all $r > 0$.
\end{lemma}

The second lemma collects two estimates whose proofs can be found in \cite{lattanzio2017gas, alves2022relaxation, alves2024weak}, being a consequence of Lemma~\ref{lemma_h}. 

\begin{lemma} \label{lemma_int_h}
Let $r$, $\bar r \in L^\gamma(\Omega)$ for some $\gamma > 1$ and assume that $0 < \bar \delta \leq \bar r(x) \leq \bar M < \infty$ for a.e. $x \in \Omega$.
Let $h(r) := \tfrac{1}{\gamma - 1} r^\gamma$. Then there exists $C>0$ such that:
\begin{enumerate}[(i)]
\item If $\gamma \geq 2 - \alpha/d$, $0 < \alpha < d$, then
\begin{equation}\label{lem_estimate_p}
\|r - \bar r \|_{\frac{2d}{d+\alpha}}^2 \leq C \int_\Omega h(r | \bar r) \, \dx.
\end{equation}
\item If $ 1 < \gamma < 2$, then
\begin{equation}\label{lem_estimate_q}
\|r - \bar r \|_{\frac{2}{3-\gamma}}^2 \leq C \int_\Omega h(r | \bar r) \, \dx.
\end{equation}

\end{enumerate}
\end{lemma}

\subsection{Proof of Theorem~\ref{thm_main}} \label{section_proof_thm}

Let $(\rho, u, n ,v)$ be a dissipative weak solution to~\eqref{BER_HF} with $\sigma > 0$ and $\gamma$ satisfying~\eqref{gamma_hyp_1} or \eqref{gamma_hyp_2}. Let $(\br,\bn)$ be a strong solution of~\eqref{BAD} such that~\eqref{bounded_away_zero} holds. \par 
First, we choose $\sigma>0$ small enough so that the relative energy $\Psi$ is nonnegative. Using the Hardy--Littlewood--Sobolev inequality \cite[Theorem~4.3]{lieb2001analysis} we have 
\begin{align*} 
\left| \io (\rho - \br - n + \bn) K_\alpha \ast (\rho - \br - n + \bn) \, \dx \right| & \leq C \, \|\rho - \br - n + \bn \|_{\frac{2d}{d+\alpha}}^2 \\
& \leq C \left( \|\rho - \br \|_{\frac{2d}{d+\alpha}}^2 + \|n - \bn \|_{\frac{2d}{d+\alpha}}^2 \right)
\end{align*}
which, by estimate \eqref{lem_estimate_p} in Lemma~\ref{lemma_int_h} yields
\begin{equation} 
\left| \io (\rho - \br - n + \bn) K_\alpha \ast (\rho - \br - n + \bn) \, \dx \right| \leq C_\ast \io h_1(\rho|\br) + h_2(n|\bn) \, \dx 
\end{equation}
for some $C_\ast > 0$. Choosing $\sigma < 2/C_\ast$ and setting $\lambda = 1-\sigma C_\ast/2 > 0$, it follows that
\begin{align*}
0  & \leq  \io h_1(\rho|\br) + h_2(n|\bn) \, \dx \\
 & \leq \frac{1}{\lambda} \io h_1(\rho|\br) + h_2(n|\bn) + \sigma \tfrac{1}{2}(\rho - \br - n + \bn) K_\alpha \ast (\rho - \br - n + \bn)  \, \dx
\end{align*}
which implies that $\Psi \geq 0$. \par 
The next step is to derive the inequality satisfied by $\Psi$, which in essence is a weak form of identity~\eqref{relative_energy_identity0}.
\begin{lemma}
For $t \in (0,T)$, the relative energy $\Psi$ satisfies
\begin{equation} \label{REL_I}
\Psi(t) - \Psi(0) + \int_0^t \io \rho|u - \bu|^2 + n|v- \bv|^2 \, \dx \, \rd \tau \leq I_1 + I_2 + I_3 + I_4
\end{equation}
where 
\begin{equation*}
\begin{split}
I_1 & = - \int_0^t \io \eps \nabla \bu : \rho (u - \bu) \otimes (u - \bar u) + \eps \nabla \bv : n (v - \bv) \otimes (v - \bv)\, \dx \, \rd \tau, \\
I_2 & = - \int_0^t \io (\nabla \cdot \bu) p_1(\rho | \br) + (\nabla \cdot \bv) p_2(n | \bn) \, \dx \, \rd \tau, \\
I_3 & = \int_0^t \io \sigma \big((\rho - \br) \bu -(n - \bn)\bv  \big) \cdot \nabla K_\alpha *(\rho - \br - n + \bn) \, \dx \, \rd \tau, \\
I_4 & = - \int_0^t \io \eps \frac{\rho}{\br} \bar e_1 \cdot (u - \bar u) + \eps \frac{n}{\bn} \bar e_2 \cdot (v - \bv) \, \dx \, \rd \tau.
\end{split}
\end{equation*}
\end{lemma}
\begin{proof}
We omit most of the details as this derivation closely follows the proof of \cite[Proposition~4.4]{alves2022relaxation} for the bipolar Euler--Poisson case, except for how one has to handle the interaction terms. In the bipolar Euler--Poisson case, we have at our disposal integration by parts formulas that do not hold here. Alternatively, one may also carry our the estimates for each species separately, following \cite{alves2024weak}, and combine the results.

We have:
\begin{align*} 
\Psi(t) - \Psi(0) \leq \ &  - \int_0^t \io \rho|u - \bu|^2 + n|v- \bv|^2 \, \dx \, \rd \tau + I_1 + I_2 + I_4 \\ 
 & - \int_0^t \io \sigma (\rho - \br - n + \bn) \partial_\tau K_\alpha \ast (\br - \bn) \, \dx \, \rd \tau \\ 
 & + \int_0^t \io \sigma \big(\nabla K_\alpha \ast (\rho - \br - n + \bn)\big) \cdot (\rho \bu - n \bv) \, \dx \, \rd \tau.
\end{align*}
Set $\xi = \rho - \br - n + \bn$. Using the symmetry of $K_\alpha$ and the no-flux boundary conditions we derive
\begin{align*}
- \io \xi K_\alpha \ast \partial_\tau (\br - \bn) \, \dx & = - \io \partial_\tau (\br - \bn) K_\alpha \ast \xi \, \dx \\
& = \io \nabla \cdot (\br \bu - \bn \bv) K_\alpha \ast \xi \, \dx \\
& = - \io (\br \bu - \bn \bv) \nabla K_\alpha \ast \xi \, \dx
\end{align*}
from which \eqref{REL_I} follows.
\end{proof}

The final step in the proof of Theorem~\ref{thm_main} consists in estimating the right-hand side of the relative energy inequality~\eqref{REL_I}
in such a way that Gronwall's inequality can be applied to conclude estimate~\eqref{main_stability}.

We estimate the term $I_1$ by
\begin{equation} \label{control_I1}
\begin{split}
I_1 & \leq C \int_0^t \io \eps \rho |u - \bu|^2 + \eps n |v - \bv|^2  \, \dx \, \rd \tau \\
& \leq C \int_0^t \Psi(\tau) \, \rd \tau.
\end{split}
\end{equation}
For $I_2$ we have
\begin{equation} \label{control_I2}
\begin{split}
I_2 & \leq C \int_0^t \io  p_1(\rho | \br) + p_2(n | \bn)   \, \dx \, \rd \tau \\
& \leq C \int_0^t \io  h_1(\rho | \br) + h_2(n | \bn)   \, \dx \, \rd \tau \\
& \leq C \int_0^t \Psi(\tau) \, \rd \tau.
\end{split}
\end{equation}
By the boundedness of the error terms $\bar{e}_1, \bar{e}_2$, the conservation of mass and hypothesis~\eqref{bounded_away_zero},
\begin{equation} \label{control_I4}
\begin{split}
I_4  & \leq \frac{1}{2 } \int_0^t \io \rho|u-\bar{u}|^2+n|v-\bar{v}|^2 \, \dx \, \rd \tau+\frac{\varepsilon^2}{2}\int_0^t \io \rho \bigg|\frac{\bar{e}_1}{\bar{\rho}} \bigg|^2+n \bigg|\frac{\bar{e}_2}{\bar{n}}\bigg|^2\, \dx \, \rd \tau\\
   & \leq \frac{1}{2 } \int_0^t \int_\Omega \rho|u-\bar{u}|^2+n|v-\bar{v}|^2 \, \dx \, \rd \tau + C \varepsilon^2 t.\\
\end{split}
\end{equation}
The term  $I_3$ is the most intricate, being also the one whose bound requires the restrictions on $\gamma$ given in \eqref{gamma_hyp_1} and \eqref{gamma_hyp_2}.
Let $J$ be given by 
\begin{equation} \label{J}
J = \io \big((\rho - \br) \bu -(n - \bn)\bv  \big) \cdot \nabla K_\alpha *(\rho - \br - n + \bn) \, \dx. 
\end{equation}
Then
\begin{equation} \label{J_bound_0}
J \leq C \io (|\rho - \br| + |n - \bn|) |I_{\alpha - 1}(\rho - \br - n + \bn)| \, \dx 
\end{equation}
where for $0 < \beta < d$, $I_\beta(f)$ denotes the classic linear fractional integral (or Riesz potential) of a measurable function $f$
\[I_\beta(f)(x) = \io f(y) |x - y|^{\beta - d} \, \dy. \]
This operator maps $L^p$ to $L^{\frac{dp}{d-\beta p}}$ for $1 < p < d/\beta$, corresponding, via duality, to the Hardy--Littlewood--Sobolev inequality; see \cite[Theorem~5.1.3]{grafakos2024fundamentals} and \cite[Theorem~4.3]{lieb2001analysis}.
\par 
We now estimate the right-hand side of \eqref{J_bound_0}, considering two separate cases according to assumptions \eqref{gamma_hyp_1} and \eqref{gamma_hyp_2}. \\ \\
\textbf{Case I}. $\gamma \geq 2$ and $1 < \alpha < d/2 + 1$:
~ \par
We first use the inequality $2ab \leq a^2 + b^2$ to estimate 
\[J \leq C \io |\rho - \br|^2 + |n - \bn|^2 \, \dx + \io |I_{\alpha - 1}(\rho - \br - n + \bn)|^2 \, \dx. \]
Now, we let $r = 2d/(d+2(\alpha - 1))$ so that
\[ \frac{dr}{d-(\alpha - 1)r} = 2. \]
From the assumptions on $\alpha$ one can readily see that 
\[ 1 < r < \min\left\{2, \frac{d}{\alpha - 1} \right\} \]
and hence, using the mapping properties of the linear fractional integral together with the embedding $L^2 \subseteq L^r$, we find
\begin{align*}
\io |I_{\alpha - 1}(\rho - \br - n + \bn)|^2 \, \dx & \leq C \left( \io |\rho - \br - n + \bn|^r \, \dx \right)^{\frac{2}{r}} \\
& \leq C \io |\rho - \br - n + \bn|^2 \, \dx \\ 
& \leq C \io |\rho - \br|^2 + |n - \bn|^2 \, \dx.
\end{align*}
Since $\gamma \geq 2$ and the strong solution is bounded away from zero we have by Lemma~\ref{lemma_h} that
\[ |\rho - \br|^2 \leq C h_1(\rho | \br) \quad \text{and} \quad |n - \bn|^2 \leq C h_2(n | \bn).\]
Putting together the above estimates yields
\begin{align*}
J &\leq C  \io |\rho - \br|^2 + |n - \bn|^2 \, \dx \\
& \leq C \io h_1(\rho | \br) + h_2(n | \bn) \, \dx \\
& \leq C \Psi,
\end{align*}
which gives
\[I_3 \leq \int_0^t \Psi(\tau) \, \rd \tau.\]
\\
\textbf{Case II}. $2 - (\alpha - 1)/d \leq \gamma < 2$ and $1 < \alpha \leq d/2 + 1$: 
~ \par 
Let \[q = \frac{2}{3-\gamma} \quad \text{and} \quad  p = \frac{2d}{d(\gamma - 1)+2(\alpha - 1)} \]
so that 
\[q^\prime = \frac{q}{q - 1} = \frac{dp}{d - (\alpha - 1)p}. \]
The assumptions on $\gamma$ and $\alpha$ imply that
\[1 < p \leq q < \gamma < 2 \quad \text{and} \quad p < \frac{d}{\alpha - 1}. \]
We then proceed using \eqref{J_bound_0}. By H\"{o}lder's inequality we have
\[J  \leq C \| |\rho - \br| + |n - \bn| \|_q \|I_{\alpha - 1}(\rho - \br - n + \bn) \|_{q^\prime},  \]
which is bounded by 
\[C (\|\rho - \br \|_q + \|n - \bn \|_q )\| \rho - \br - n + \bn\|_p \]
due to the mapping properties of the linear fractional integral. Since $p \leq q$ and $\Omega$ has finite measure, the previous quantity is controlled by
\[C (\|\rho - \br \|_q + \|n - \bn \|_q )\| \rho - \br - n + \bn\|_q, \]
which, in turn, is bounded by 
\[C (\|\rho - \br \|_q^2 + \|n - \bn \|_q^2). \]
We now use estimate~\eqref{lem_estimate_q} in Lemma~\ref{lemma_int_h} to conclude that 
\[\|\rho - \br \|_q^2 \leq C \io h_1(\rho | \br) \, \dx \quad \text{and} \quad \|n - \bn \|_q^2 \leq  C \io h_2(n | \bn) \, \dx.\]
It follows that $J \leq C \Psi$ and hence 
\[I_3 \leq \int_0^t \Psi(\tau) \, \rd \tau.\] 
\par 
Finally, combining the bounds derived so far for the terms $I_1, \ldots, I_4$ with the relative energy inequality~\eqref{REL_I} gives
\[\Psi(t) - \Psi(0) \leq C \int_0^t \Psi(\tau) \, \rd \tau + C \eps^2 t, \]
from which \eqref{main_stability} follows by Gronwall's lemma.

\section*{Acknowledgments}
This research was partially funded by the Austrian Science Fund (FWF), project number 10.55776/F65.

\end{document}